\documentclass[11pt]{amsart}
\usepackage{geometry}                		% See geometry.pdf to learn the layout options. There are lots.
\usepackage{graphicx}				% Use pdf, png, jpg, or eps§ with pdflatex; use eps in DVI mode
\usepackage[small,nohug,heads=vee]{diagrams}
\diagramstyle[labelstyle=\scriptstyle]															
\usepackage{amssymb}
\usepackage{amsmath,amscd}
\usepackage{amsthm}
\usepackage{mathrsfs}
\usepackage{epstopdf}
\theoremstyle{theorem}
\newtheorem{theorem}{Theorem}[section]
\theoremstyle{theorem}

\theoremstyle{theorem}
\newtheorem{proposition}[theorem]{Proposition}
\theoremstyle{remark}
\newtheorem{remark}[theorem]{Remark}
\theoremstyle{theorem}
\newtheorem{definition}[theorem]{Definition}
\theoremstyle{theorem}

\newcommand{\Hom}{\text{Hom}\,}
\newcommand{\Ext}{\text{Ext}\,}
\newcommand{\ob}{\text{ob}\,}
\newcommand{\Id}{\text{Id}}
\newcommand{\perf}{\text{Perf}\,}
\newcommand{\Stab}{\text{Stab}\,}
\newcommand{\Arg}{\text{Arg}\,}
\newcommand{\Supp}{\text{Supp}\,}
\newcommand{\prim}{\text{prim}}
\newcommand{\Li}{\text{Li}_2\,}
\newcommand{\ud}{\mathrm{d}}

\begin{document}
\title{Donaldson--Thomas invariants and wall--crossing formulas}
\author{Yuecheng Zhu}
%\institute{University of Texas at Austin, Mathematics Department, 2512 Speedway Stop C1200, Austin, TX, 78712, the United States; email: yuechengzhu@math.utexas.edu.}
\maketitle
%\section{}
%\subsection{}
\begin{abstract}
We introduce the Donaldson--Thomas invariants and describe the wall--crossing formulas for numerical Donaldson-Thomas invariants. 
\end{abstract}
%\keywords{Donaldson--Thomas invariants Wall-crossing formulas}}
\section{Introduction}
This is a brief introduction to the Donaldson--Thomas invariants and wall--crossing formulas, based on a talk given by the author at the Fields institute. The standard references are \cite{KS10}, \cite{KS08}, and \cite{KS11}. We will focus on a few basic definitions and ideas. It doesn't intend to be a comprehensive introduction to the vast program by Kontsevich--Soibelman. There are also many other works on the subject, e.g., Joyce--Song's program \cite{JS}, that we are not able to touch here.  \\

First, we want to give a glimpse of a much bigger picture that can not be included in this paper. The wall--crossing formulas (WCF) we will introduce are simply certain identities in the group of automorphisms of an algebraic torus. However, they are satisfied by a wide range of numerical invariants from very different problems. These numerical invariants are,
\begin{enumerate}
\item Donldson--Thomas (DT) invariants for $3$-Calabi--Yau categories,
\item Gromov--witten (GW) type invariants that are used in a scattering diagram,
\item some invariants produced from a complex integrable system. 
\end{enumerate}

There are other wall--crossing formulas originated from Physics. The invariants are the counting of BPS states in several different supersymmetric quantum systems. For example, in $d=2, \mathcal{N}=(2,2)$ theories, the WCF is called Cecotti--Vafa formula. See \cite{CV} and \cite{GMN}. These formulas are very closely related to DT invariants for $3$-Calabi--Yau categories. For a good introduction to the story, see the slides~\cite{Nslides1}. \\

We will only talk about the first case, DT invariants. The second invariants and scattering diagrams were originated from Kontsevich and Soibelman's work on mirror symmetry \cite{KS04}.  They are developed and applied in Gross--Siebert's program to reconstruct mirror families. It is worth noting that these Gromov--Witten type invariants look very different from DT invariants, and there are no stability conditions involved in this case. A good introduction is \cite{GPS}. The invariants from integrable systems are introduced in \cite{KS13}. The amazing fact is that all these invariants satisfy WCF. \\

The fact that the same (or similar) formula(s) appear in many different setups suggests that there is a common structure behind all these different setups. This is indeed the case. In \cite{KS13}, this structure is introduced, and is called the wall crossing structure (WCS). It is not hard to see that WCF and WCS are very central to enumerative geometry. \\

\subsection{Acknowledgement}
The author would like to thank the Fields institute for the hospitality. The author would also like to thank Andrew Neitzke for explaining the quadratic refinement, and the referee for important corrections and suggestions. 
%\end{acknowledgement}

\section{Donaldson--Thomas invariants}\label{1}
\subsection{$3$-Calabi--Yau categories}
This part is taken from \cite{Ke}
Fix a base field $k$. For any category (or $A_\infty$-category) $\mathcal{C}$, the set of objects is denoted by $\ob(\mathcal{C})$. For any $E,F \in \ob(\mathcal{C})$, the morphism from $E$ to $F$ is denoted by $\mathcal{C}(E,F)$. All categories are assumed to be $k$-linear. \\

For a Calabi--Yau $d$-fold $X$ over $k$, since the canonical bundle is trivial, the Serre duality gives a non-degenerate pairing 
\[
\Ext^i(\mathscr{E},\mathscr{F})\times \Ext^{d-i}(\mathscr{F},\mathscr{E})\to k,
\]

for all coherent sheaves $\mathscr{E}$ and $\mathscr{F}$ on $X$. \\

Let $\mathcal{D}^b(X)$ be the bounded derived category of coherent sheaves on $X$. It inherits a non-degenerate pairing from the above, which is denoted by $(\cdot,\cdot)$:
\[
\mathcal{D}^b(X)(\mathscr{E},\mathscr{F})\otimes \mathcal{D}^b(X)(\mathscr{F},\mathscr{E})\to k[-d], \quad  \mathscr{E},\mathscr{F}\in \mathcal{D}^b(X). 
\]

If $k=\mathbf{C}$, one considers the dg-model $\perf(X)$ of $\mathcal{D}^b(X)$. It is a thick triangulated subcategory generated by perfect complexes, i.e., those quasi-isomorphic to bounded complexes of finite rank vector bundles. It is a dg-module over the dg-algebra $\Omega^{0,*}(X)$. This category $\perf(X)$ for a Calabi--Yau $3$-fold $X$ is the model for the $3$-Calabi--Yau category we are interested in.\\

Let $\mathcal{T}$ be a triangulated $k$-category, which is Hom--finite, i.e., for any two objects $E,F\in\ob(\mathcal{T})$, the morphism space $\mathcal{T}(E,F)$ is a finite dimensional $k$-vector space. For a triangulated category $\mathcal{T}$, we can always assume the suspension functor $[1]$ is an automorphism, instead of just an auto equivalence. We always write a triangle functor as a pair $(S,\iota)$, where $\iota$ is the isomorphism of functors $S[1]\to [1]S$. For any $k$-vector space $V$, its dual space is denoted by $V^*$. 

\begin{definition}
A triangle functor $(S,\iota): \mathcal{T}\to \mathcal{T}$ is called a right Serre functor, if there exists a family of isomorphisms of functors (called the trace maps)
\[
t_E: \mathcal{T}(\cdot, SE)\to \mathcal{T}(E,\cdot)^*, 
\]

functorial in $E\in \mathcal{T}$, and that for all $E, F\in \mathcal{T}$, the following diagram commutes
\[
\begin{diagram}
\mathcal{T}([1]F,S[1]E)&\rTo^{t_{[1]E}}&&\mathcal{T}([1]E,[1]F)^*\\
\dTo^\iota&&&\dTo_{-[-1]^*}\\
\mathcal{T}([1]F,[1]SE)&\rTo^{[-1]}&\mathcal{T}(F,SE)&\rTo^{t_E}\mathcal{T}(E,F)^*.
\end{diagram}
\]

A right Serre functor is called a Serre functor, if it is an auto equivalence. In this case, we say $\mathcal{T}$ has the Serre duality. 
\end{definition}

If $X/k$ is a smooth projective variety of dimension $d$, and $\omega_X$ is the canonical sheaf of $X$, the functor
\[
S: \mathscr{F}\mapsto \mathscr{F}\otimes \omega_X[d]
\]

is a Serre functor for $\mathcal{D}^b(X)$ (or there exists a natural transformation $\iota$ such that $(S,\iota)$ is a Serre functor). This is the content of the usual Serre duality in algebraic geometry. So it is not hard to imagine the definition of a $d$-Calabi--Yau category. Basically, we want $S\cong [d]$. For any triangulated category $\mathcal{T}$, there is a natural antomorphism $([1], -\Id_{[2]})$, where $-\Id_{[2]}$ is the negative of the identity
\[
[2]=[1][1]\to [1][1]=[2]. 
\]

The negative sign is necessary to make it a triangle functor. For example, when $\mathcal{T}$ is the derived category of an abelian category,  the functor $[1]$ changes the differential of a chain complex from $\partial$ to $-\partial$. 

\begin{definition}
A triangulated $k$-category $\mathcal{T}$ is called a $d$-Calabi--Yau category, if it admits a Serre functor $(S,\iota)$ and there is an isomorphism of triangle functors, 
\[
(S,\iota)\cong ([1],-\Id_{[2]})^d. 
\]
\end{definition} 

The following proposition makes it more clear that the definition is the right generalization of $\mathcal{D}^b(X)$ for a Calabi--Yau $d$-fold $X/k$. 

\begin{proposition}
Suppose the triangulated $k$-category $\mathcal{T}$ admits a Serre functor. $\mathcal{T}$ is $d$-Calabi--Yau if and only if there is a family of linear forms
\[
t_E: \mathcal{T}(E,[d]E)\to k, \quad E\in\mathcal{T},
\]

such that for all objects $E$ and $F$, the induced pairing 
\begin{eqnarray*}
(\cdot,\cdot): &\mathcal{T}(E,F)\times \mathcal{T}(F,[d]E)\to k\\
&(f,g)\mapsto t_E(f\circ g),
\end{eqnarray*}

is non-degenerate, and for all morphisms $g: E\to [p]F$ and $f: F\to [q]E$ with $p+q=d$,
\[
t_E(([p]f)\circ g)=(-1)^{pq}t_F(([q]g)\circ f).
\]
\end{proposition}

%The proof? \\

For any $E\in \mathcal{T}$, recall the graded algebra
\[
A_E:=\Ext^*(E,E)=\bigoplus_{p\in\mathbf{Z}}\mathcal{T}(E,[p]E).
\]

If $f$ and $g$ are homogeneous elements, and $g$ is of degree $p$, the multiplication $f\cdot g$ is defined to be $([p]f)\circ g$. Suppose $\mathcal{T}$ is $d$-Calabi--Yau, then we can define the linear form
\[
t: A_E\to k
\]

which is 
\[
t_E: \Ext^d(E,E)\to k,
\]

on $A_E^d$ and zero on any other degree. The proposition implies that the pairing
\[
(a,b)=t(a\cdot b)
\]

is non-degenerate and supersymmetric. \\

After Kontsevich and Soibelman, people should consider $A_\infty$ categories. In that case, this non-degenerate pairing $(\cdot,\cdot)$ is what characterizes the Calabi--Yau property.\\

Let $\mathcal{A}$ be a minimal $A_\infty$-category ($m_1=0$) over $k$, whose morphism spaces are of finite total dimension, and $d$ be a positive integer. 
\begin{definition}
A cyclic structure of degree $d$ on $\mathcal{A}$ is the datum of bilinear forms
\[
(\cdot,\cdot): \mathcal{A}(E,F)\times \mathcal{A}(F,E)\to k
\]

of degree $-d$ such that
\begin{enumerate}
\item $(\cdot,\cdot)$ is non-degenerate for all $E,F\in \mathcal{A}$.
\item For any $n>0$ and all $E_0,E_1\ldots, E_n$, the map
\[
w_{n+1}: \mathcal{A}(E_{n-1},E_n)\otimes\mathcal{A}(E_{n-2}, E_{n-1})\otimes\ldots\otimes\mathcal{A}(E_0,E_1)\otimes \mathcal{A}(E_n,E_0)\to k
\]

defined by 
\[
(a_{n_1}, a_{n-2}, \ldots, a_0, a_n)\mapsto (m_n(a_{n_1}, a_{n-2}, \ldots, a_0), a_n)
\] 

is cyclically invariant, i.e. we have 
\[
w_{n+1}(a_{n-1},a_{n-2},\ldots,a_0,a_n)=\pm w_{n+1}(a_n, a_{n-1},a_{n-2},\ldots,a_0).
\]

Here the sign $\pm$ depends on $n$ and the parities of the homogeneous elements $a_i$. 
\end{enumerate}
\end{definition}

For any $A_\infty$ category $\mathcal{A}$, one can define the perfect derived category $\perf(\mathcal{A})$ as the thick triangulated subcategory of the derived category $\mathcal{D}(\mathcal{A})$ generated by the representable $A_\infty$ modules $\mathcal{A}(\cdot, X)$ for $X\in \mathcal{A}$. One can show that $\perf(\mathcal{A})$ is Hom--finite. 

\begin{proposition}
If $\mathcal{A}$ has a cyclic structure of degree $d$, then $\perf(\mathcal{A})$ is a $d$-Calabi--Yau category in the sense of the definition given earlier. 
\end{proposition}

Therefore, from now on, by a $3$-Calabi--Yau category $\mathcal{C}$, we mean a triangulated $A_\infty$ category $\mathcal{C}$ with a cyclic structure of degree $3$. This is also called a non-commutative Calabi--Yau variety of dimension $3$ by Kontsevich and Soibelman, and is the natural setting for DT invariants, if you want to consider all the interesting examples. \\

For $\mathcal{C}$ a $3$-Calabi--Yau category, and $E$ an object in $\mathcal{C}$. Define the potential $W_E$ as the formal power series  
\[
W_E(a)=\sum_{n\geqslant 1}\frac{w_{n+1}(a,\dots, a)}{n+1}
\]

for $a\in \Ext^1(E,E)$. Here we have used the assumption that $\mathcal{C}$ is minimal. In general, $W_E$ induces a formal function $W_E^{\text{min}}$ over $\Ext^1(E,E)$. 

\subsection{Bridgeland's Stability conditions and DT invariants}
The natural triangulated category $\mathcal{T}$ itself is usually too big. We need to use some stability conditions to chop the category down to manageable size. In history, various notions of stability have been studied for the category of sheaves on a variety. The following stability condition is introduced by Bridgeland in \cite{Br07} for a general triangulated category. 

\begin{definition}
A stability condition $\sigma=(Z,\mathcal{P})$ on a triangulated category $\mathcal{T}$ consists of a group homomorphism $Z: K(\mathcal{T})\to \mathbf{C}$ called the central charge, and a collection of full additive subcategories $\{\mathcal{P}(\phi)\}$ for each $\phi\in \mathbf{R}$, satisfying the following axioms:
\begin{enumerate}
\item if $E\in \mathcal{P}(\phi)$, then $Z(E)=m(E)\exp(i\pi \phi)$ for some $m(E)\in \mathbf{R}_{>0}$,
\item  for all $\phi\in\mathbf{R}$, $\mathcal{P}(\phi+1)=\mathcal{P}(\phi)[1]$,
\item if $\phi_1>\phi_2$, and $E_i\in \mathcal{P}(\phi_i)$, then $\mathcal{T}(E_1,E_2)=0$,
\item for each nonzero object $E\in \mathcal{T}$, there are a finite sequence of real numbers
\[
\phi_1>\phi_2>\ldots>\phi_n,
\]

and a collection of exact triangles
\[
\begin{diagram}
0=&E_0 &\rTo &E_1&\rTo&E_2&\rTo& \ldots&\rTo &E_{n-1}&\rTo &E_n&=E,\\
&&\luDashto(1,2)\ldTo(1,2)&&\luDashto(1,2)\ldTo(1,2)&&&&&&\luDashto(1,2)\ldTo(1,2)\\
&&A_1&&A_2&&&&&&A_n
\end{diagram}
\]

with $A_i\in\mathcal{P}(\phi_i)$ for all $i$. 
\end{enumerate}
\end{definition}

Let $X$ be a smooth projective curve over $k$, and $\mathcal{T}=\mathcal{D}^b(X)$. Choose the heart to be the full sub category of coherent sheaves over $X$. For any coherent sheaf $\mathscr{F}$, define the central charge to be the slope
\[
\mu(\mathscr{F})=\frac{\deg \mathscr{F}}{\text{rank}\,{\mathscr{F}}}. 
\]

The heart, with the central charge on the heart, induce a unique stability condition on $\mathcal{D}^b(X)$. This is called Mumford's stability condition, and is one of the most famous stability conditions.\\

The set of all stability conditions is denoted by 
\[
\Stab(\mathcal{T}):=\{\sigma=(Z,\mathcal{P})\}. 
\]

It is an important theorem of Bridgeland that this space $\Stab(\mathcal{T})$ can be endowed with a natural topology such that locally, the map $\Stab(\mathcal{T})\to \Hom(K(\mathcal{T}),\mathbf{C})$,
\[
\sigma=(Z,\mathcal{P})\mapsto Z,
\]

is a homeomorphism onto the image.

\begin{remark}
To define the stability for sheaves on a higher-dimensional variety $X$, a polarization by an ample line bundle is needed. The following definition of stability is given by Simpson. Fix an ample line bundle $\mathscr{H}$, define the normalized Hilbert polynomial for every coherent sheaf $\mathscr{E}$,
\[
P_{\mathscr{H},\mathscr{E}}(n):=\frac{1}{\text{rank}\,\mathscr{E}}\chi(\mathscr{E}\otimes\mathscr{H}^n).
\]

Then $\mathscr{E}$ is called Gieseker stable (resp., semi-stable), if for all coherent subsheaves $\mathscr{F}\subset\mathscr{E}$ with $0<\text{rank}\,\mathscr{F}<\text{rank}\,\mathscr{E}$, we have $P_{\mathscr{H},\mathscr{F}}(n)<P_{\mathscr{H},\mathscr{E}}(n)$ (resp., $\leqslant$) for $n\gg 0$.  However, Gieseker's stability is not an example of Bridgeland's stability condition, but a limit of the space $\Stab(\mathcal{D}^b(X))$\footnote{A limit at the infinity.}. Note that we need the topology on $\Stab(\mathcal{D}^b(X))$ to talk about the limit. 
\end{remark}

Since the Grothendieck group $K(\mathcal{T})$ is usually not finitely generated, in order to get a finite-dimensional complex manifold, people usually assume the central charge $Z:K(\mathcal{T})\to \mathbf{C}$ factors through a finitely generated free abelian group $\Gamma\to \mathbf{C}$. For example, if $\mathcal{T}=\mathcal{D}^b(X)$ for a complex projective variety $X$, $\Gamma$ can be the numerical Chow group (use numerical equivalence) or the Betti cohomology. In general we can require the following condition. Suppose $\mathcal{T}$ is of finite type, that is for every pair of objects $E,F$ of $\mathcal{T}$, the $k$-vector space $\bigoplus_i \Ext^i(E,F)$ is finite dimensional. Then one defines the Euler form
\[
\chi(E,F):=\sum_i (-1)^i \dim \Ext^i(E,F).
\]

The free abelian group $N(\mathcal{T}):=K(\mathcal{T})/K(\mathcal{T})^\perp$, where $K(\mathcal{T})^\perp$ means the orthogonal complement with respect to the Euler form, is called the numerical Grothendieck group of $\mathcal{T}$. If this group $N(\mathcal{T})$ is of finite rank, the category $\mathcal{T}$ is said to be numerically finite. Assume $\mathcal{T}$ is of finite type, and numerically finite, define $\Stab_N(\mathcal{T})$ to be the subspace of $\Stab(\mathcal{T})$ consisting of stability conditions, for which the central charge $Z$ factors through $K(\mathcal{T})\to N(\mathcal{T})$. Restricting the natural topology of $\Stab(\mathcal{T})$ to $\Stab_N(\mathcal{T})$, $\Stab_N(\mathcal{T})$ is a finite--dimensional complex manifold (not necessarily connected). \\

From now on, we always assume that a central charge will factor through some finitely generated free abelian group $\Gamma$. We still use the notation $\Stab(\mathcal{T})$ for simplicity. Following the example of $\mathcal{D}^b(X)$, we call an element $\gamma\in \Gamma$ a Chern character. We focus on the case when the triangulated category $\mathcal{C}$ is a $3$-Calabi--Yau category. \\

Fix a stability condition $\sigma=(Z,\mathcal{P})$ and a Chern character $\gamma\in \Gamma$, construct the moduli space $\mathcal{M}^{\text{\ss}}_\sigma(\gamma)$\footnote{Since ''\ss'' is "ss" in German, we use it to stand for semi-stable.} of semi-stable objects with Chern character $\gamma$. It contains the stable objects as a dense open subspace. Our goal is to "count" the objects in the space $\mathcal{M}^{\text{\ss}}_\sigma(\gamma)$.\\

Let's first consider an ideal situation. Assume that the coarse moduli space exists, and is a projective variety. We want to define some counting invariants. One might first attempt to associate to $\mathcal{M}_\sigma^{\text{\ss}}(\gamma)$ its Euler characteristic $\chi(\mathcal{M}^{\text{\ss}}_\sigma(\gamma))$. But this naive approach doesn't usually work. A sensible, deformation invariant counting in geometry requires a virtual fundamental class.  The existence of the virtual fundamental class depends on the obstruction theory, not just on the moduli space as a scheme. There is a type of obstruction theory called the perfect obstruction theory that produces a virtual fundamental class. A perfect obstruction theory is called symmetric, if the deformation space and the obstruction space are dual to each other. In this case, the virtual fundamental class is of degree $0$. Suppose there is a symmetric obstruction theory for $\mathcal{M}^{\text{\ss}}_\sigma(\gamma)$. The integral of $1$ against the virtual fundamental class can be regarded as the virtual counting of points in $\mathcal{M}^{\text{\ss}}_\sigma(\gamma)$. This is the DT invariant for the Chern character $\gamma$, and is denoted by $\Omega_\sigma(\gamma)$. The moduli space doesn't have to be from a Bridgeland's stability. In \cite{Br}, any virtual count of a proper scheme with a symmetric obstruction theory is called a Donaldson--Thomas type invariant. We simply call them classic DT invariants to distinguish them from general ones.  \\

Suppose $X$ is a Calabi--Yau $3$-fold, choose $\gamma=(1,0,-\beta,-n)\in H^0\oplus H^2\oplus H^4\oplus H^6$. Denote the space of the moduli space of Gieseker stable sheaves with trivial determinant and Chern character $\gamma$ by $I_n(X,\beta)$. This is a fine moduli space (Hilbert scheme). The deformation/obstruction complex has amplitude in degree $1$ and $2$. The Serre duality for $X$ implies that $I_n(X,\beta)$ admits a symmetric obstruction theory.  Moreover, $I_n(X,\beta)$ is equal to the space of semi-stable objects. Define
\[
I_{n,\beta}:=\int_{[I_n(X,\beta)]^{\text{vir}}}1.
\]

$I_{n,\beta}$ is always an integer. This is the original Donaldson--Thomas invariant studied in \cite{DT98} and \cite{T00}. The existence of a symmetric obstruction theory is the reason why people are most interested in counting curves on a Calabi--Yau $3$-fold, as opposed to a general dimension. If a Bridgeland's stability condition $\sigma$ is very close to the Gieseker stability, then we have the classic DT invariants $\Omega_\sigma(\gamma)$ constructed from symmetric obstruction theories, and $\Omega_\sigma(\gamma)=I_{n,\beta}$. That is how the invariants $\Omega_\sigma(\gamma)$ got the names. \\

In general, however, the moduli space $\mathcal{M}_\sigma^{\text{\ss}}(\gamma)$ is an Artin stack, and we don't have a symmetric obstruction theory. A different approach is needed to define the general DT invariants. In \cite{Br}, Behrend discovered an equivalent definition of the classic DT invariants. Note that if the moduli space $\mathcal{M}$ is smooth, and the obstruction bundle is $\Omega_\mathcal{M}$, the dual of the tangent bundle, then the virtual counting is $(-1)^{\dim \mathcal{M}}\chi(\mathcal{M})$. In general, as proved in loc. cit., the virtual counting from the symmetric obstruction theory is equal to a weighted Euler characteristics $\chi(\mathcal{M},\nu_\mathcal{M})$, for a weight function $\nu_\mathcal{M}$
\[
\chi(\mathcal{M},\nu_\mathcal{M})=\sum_{n\in\mathbf{Z}}n\chi(\{\nu_\mathcal{M}=n\}),
\]

where $\chi$ on the right hand side is the Euler characteristic of the discrete set. The weight function $\nu_\mathcal{M}$ can be  heuristically interpreted as follows. First when $\mathcal{M}$ is a critical locus of a regular function $f$ over a smooth ambient space $U$, we have
\[
\nu_\mathcal{M}(p)=(-1)^{\dim U}(1-\chi(MF_p)),
\] 

where $MF_p$ is the Milnor fibre at the point $p\in \mathcal{M}\subset U$. If a space $\mathcal{M}$ admits a symmetric obstruction theory, then it should be viewed, at least locally, as the critical locus of a regular functional $f$ over some smooth ambient space $U$\footnote{For example, in the original DT invariants defined by Thomas, the moduli space is the set of critical points of the holomorphic Chern-Simons functional.}. Therefore, heuristically, the DT invariants are defined in terms of Milnor fibers. Moreover, the use of the Euler characteristic suggests that they are from some motives.\\

This approach to classic DT invariants can be generalized. Morally, we should view a $3$-Calabi--Yau category $\mathcal{C}$ as follows. The objects $\ob(\mathcal{C})$ form a set. The morphisms form a bundles over the sets $\ob(\mathcal{C})\times\ldots\times\ob(\mathcal{C})$, and higher composition maps are morphisms of tensor products of such bundles. Then we have the bundle over $\ob(\mathcal{C})$ defined by the space $\Ext^1(E,E)$, and the formal function $W_E^{\min}$ over the bundle near the zero section. Consider the Milnor fiber of the potential $W_E^{\min}$, and define the weights $\nu$ by taking the Euler characteristics. Then the integral of the weight function $\nu$ against certain measure on $\mathcal{M}^{\text{\ss}}_\sigma(\gamma)\subset \ob(\mathcal{C})$ should be the DT invariants.\\

Of course, this is just a very rough idea. The general theory in \cite{KS08} is quite involved, and we are unable to get into the details. First, in order to make sense of the integral, you need some control of the category. The $3$-Calabi--Yau category $\mathcal{C}$ is assumed to be ind-constructible, so that the set $\ob(\mathcal{C})$ is an ind-constructible set, and the bundles are constructible bundles. The theory of Bridgeland's stability condition is modified correspondingly. Secondly, the whole theory is motivic. In loc. cit., the theory of motivic stack functions is developed, and for each ind-constructible category $\mathcal{C}'$, the motivic Hall algebra $H(\mathcal{C}')$ is defined. For any stability condition $\sigma$ and any strict sector\footnote{This is used to talk about formal functions, and we will use it again in the next section.} $V\subset \mathbf{R}^2$, pick a branch of the log function, we can define a category $\mathcal{C}_V$. The Milnor fiber of $W_E^{\min}$ is replaced by the motivic Milnor fiber. The motivic weight is defined by the motivic Milnor fiber with some additional data called the orientation data. The measure is defined as an invertible element in the motivic Hall algebra $\widehat{H}(\mathcal{C}_V)$. For each $V$, the integral takes the value $A_V^{\text{mot}}$ in a unital associative algebra $\mathcal{R}_V$ called the motivic quantum torus. The motivic DT invariant is thus defined as a collection $\{A_V^{\text{mot}}\}$ (one for each $V$) for an ind-constructible $3$-Calabi--Yau category $\mathcal{C}$ with a stability condition $\sigma$, under the assumption of a conjectural integral identity and the assumption of the existence of orientation data. The integral identity conjecture (loc. cit. Sect. 4.4 Conjecture 4) is proved for the $l$-adic realization of the motive (loc. cit. Sect. 4.4 Proposition 9). For the construction of the numerical DT invariants, it suffices.\\

The way to get numerical DT invariants from motivic ones is by taking the quasi-classic limit. First, by a twisted Serre polynomial, we have a realization from the motivic quantum tori to quantum tori. The motivic DT invariants are then described in terms of automorphisms of the quantum tori. The quasi-classic limit of the integer quantum torus is a Poisson torus which we will see later. Assume the absence of poles conjecture, the quasi-classic limit of the automorphisms exist, and we get numbers which are defined to be the numerical DT invariants. It is not obvious that these numbers are integers. It is a conjecture that they should be integers. Moreover, in certain cases, it's been proved that the numerical DT invariants thus defined agree with the classic DT invariants. Therefore, the motivic DT invariants should be regarded as the quantization of the classic DT invariants, with the quantization parameter being the motive of an affine line.\\

%There is another approach of generalizing DT invariants due to Joyce and Song. Recall the weight $\nu_\mathcal{M}(p)$ in Behrend's weighted Euler characteristic is heuristically defined in terms of the Euler characteristic of the Milnor fiber $MF_p$. Then it should be regarded as the Euler characteristic of the perverse sheaf of the vanishing cycle of the holomorphic functional $f$. This idea can be made rigorous with the introduction of the d-critical locus. Joyce with his collaborators have constructed categorification of DT invariants for d-critical loci with an orientations.\\

From now on, let's assume the numerical DT invariants $\Omega_\sigma(\gamma)$ are defined for generic stability conditions.\\

For any $\mathbf{Q}$-linearly independent collection of vectors $\{\gamma_1,\ldots,\gamma_k\in \Gamma\}$ with $k\geqslant 2$, and for a general $\sigma=(Z,\mathcal{P})\in \Stab(\mathcal{C})$, the homomorphism $Z$ restricted to the $\mathbf{R}$-linear span of $\{\gamma_1,\ldots,\gamma_k\}$ is surjective onto $\mathbf{C}$. In other words, the $k$ vectors $\gamma_1,\ldots,\gamma_k$ are not mapped to a straight line in $\mathbf{C}$ by $Z$. \\

\begin{definition}
The subset
\[
\Big\{\sigma=(Z,\mathcal{P})\in \Stab(\mathcal{C}): \exists \gamma_1,\gamma_2\in \Gamma, \mathbf{Q}-\text{linealy independent, with } \Arg(Z(\gamma_1))=\Arg(Z(\gamma_2))\Big\}
\]

is called a wall. 
\end{definition}

The wall is a countable union of real codimension $1$ strata. The DT invariant $\Omega_\sigma(\gamma)$ is locally constant, if $\sigma$ stays away from the wall. However, If we go along a path of stability conditions $\sigma_t=(Z_t,\mathscr{P}_t)$ that crosses the wall, $\Omega_{\sigma_t}(\gamma)$ would jump. This phenomenon is called the wall--crossing phenomenon. It is very important to find out the transformation rule for $\{\Omega_\sigma(\gamma)\}$ when the stability conditions cross the wall. For example, if the limit stability conditions are included, then the conjectured PT/DT correspondence can be understood as wall crossings. See \cite{Ba09} and \cite{To09}. It turns out when $\mathcal{C}$ is $3$-Calabi--Yau, there is an additional structure on the generating functions $\sum_{\gamma\in \Gamma} \Omega_\sigma(\gamma)e_\gamma$ that express the transformation rule nicely. This is the wall--crossing formula (WCF) we are going to turn to. 

\section{Wall--crossing formulas}\label{2}
We need to define the stability data for a graded Lie algebra, which is analogous to the stability condition for a triangulated category. \\

Fix $k$ and a free abelian group $\Gamma$ of finite rank. Let $\mathfrak{g}$ be a Lie algebra over $k$ graded by $\Gamma$.
\[
\mathfrak{g}=\bigoplus_{\gamma\in \Gamma}\mathfrak{g}_\gamma.
\]

\begin{definition}
A stability data on $\mathfrak{g}$ is a pair $\sigma=(Z,a)$ such that
\begin{enumerate}
\item $Z: \Gamma\to \mathbf{C}$ is a group homomorphism,
\item $a=\{a(\gamma)\}_{\gamma\in \Gamma}$ is a collection of elements $a(\gamma)\in \mathfrak{g}_\gamma$, satisfying the following property called the support property: 
There exists a non-degenerate quadratic form $Q$ on $\Gamma_\mathbf{R}$ such that
\begin{itemize}
\item  $Q\vert_{\ker Z}<0$, where we use the same notation $Z$ for the natural extension of $Z$ to $\Gamma_\mathbf{R}$,
\item  $\Supp a:=\{\gamma\in \Gamma: a(\gamma)\neq 0\}\subset \{\gamma\in \Gamma\backslash \{0\}: Q(\gamma)>0\}$.
\end{itemize}
\end{enumerate}
\end{definition}

\begin{remark}
The support property needs explanation. It is equivalent to the following property. There exists a norm $\Vert \cdot\Vert$ on $\Gamma_\mathbf{R}$ and a constant $C>0$ such that for any $\gamma\in \Supp a$, one has 
\[
\Vert\gamma\Vert\leqslant C \vert Z(\gamma)\vert. 
\]

The equivalence follows from the following relation between the quadratic form $Q$ and the norm $\Vert\cdot\Vert$,
\[
Q(\gamma)=-\Vert\gamma\Vert^2+C'\vert Z(\gamma)\vert^2. 
\]

We call both of them the support property. It is clear from the second formulation that the support property implies the image of $\Supp a$ under $Z$ is discrete in $\mathbf{C}$, with at most polynomially growing density at infinity. It is related to the locally finiteness of the stability conditions in the sense of Bridgeland.  
\end{remark}

The stability data $(Z,a)$ is equivalent to an equivalent class of a triple $(Z,Q,A)$. Let $\mathcal{S}$ be the set of strict cone sectors in $\mathbf{R}^2$, possibly degenerate (i.e. rays). \\

Consider the triple $(Z,Q,A)$ such that,
\begin{enumerate}
\item $Z:\Gamma\to \mathbf{R}^2$ is a group homomorphism (extended to $\Gamma_\mathbf{R}$ linearly),
\item $Q$ is a non-degenerate quadratic form on $\Gamma_\mathbf{R}$ such that $Q\vert_{\ker Z}<0$,
\item $A$ is an assignment $V\to A_V\in G_{V,Z,Q}$, where $V\in \mathcal{S}$ and $G_{V,Z,Q}$ is a pronilpotent group whose Lie algebra is
\[
\mathfrak{g}_{V,Z,Q}:=\prod_{\gamma\in C(V,Z,Q)\cap \Gamma}\mathfrak{g}_\gamma,
\]

and $C(V,Z,Q)$ is a convex cone in $\Gamma_\mathbf{R}$ generated by $Z^{-1}(V)\cap \{Q\geqslant 0\}$. The assignment $A$ is required to satisfy the factorization property: If $V=V_1\coprod V_2$ clockwise, then $A_V=A_{V_1}\cdot A_{V_2}\in G_{V,Z,Q}$. 
\end{enumerate}

There is an equivalence relation on the set of all triples $(Z,Q,A)$. We say $(Z,Q,A)$ is equivalent to $(Z', Q', A')$ if $Z=Z'$, $A_V$ and $A'_V$ can be identified as an element in some pronilpotent group $G_{V,Z,Q''}$ for every $V\in \mathcal{S}$.

\begin{theorem}
There is a bijection between the the set of equivalent classes of the triple $(Z, Q, A)$ and the set of stability data $(Z,a)$. 
\end{theorem}

\begin{proof}
The key is the factorization property. If $(Z,a)$ is given, for every ray $l\subset \mathbf{R}^2$, define 
\[
A_l:=\exp{\Bigg(\sum_{Z(\gamma)\in l, Q(\gamma)\geqslant 0}a(\gamma)\Bigg)}. 
\]

Then, for arbitrary $V\in \mathcal{S}$, 
\[
A_V=\prod^{\longrightarrow}_{l\subset V} A_l. 
\]

We use "$\longrightarrow$" to denote the clockwise product, and "$\longleftarrow$" to denote the counterclockwise product.\\ 

On the other hand, for a generic $(Z,Q,A)$, $Z$ is injective on $\Gamma$. We can read the data $a(\gamma)$ off from each $A_l$ by taking log, since $A_l$ are in pronilpotent groups. 
\end{proof}

Therefore, we also call the triple $(Z, Q,A)$ a stability data. Define the space of stability data for $\mathfrak{g}$, 
\[
\Stab(\mathfrak{g}):=\big\{\sigma=(Z,a) \text{ is a stability data}\big\}.
\]

The magic thing about the definition is that $\Stab(\mathfrak{g})$ is also endowed with a natural topology. Let $X$ be a topological space, $x_0\in X$ a point, and $(Z_x, a_x)$ a family of stability data parametrized by $X$. 

\begin{definition}
The family $(Z_x,a_x)$ is called continuous at $x_0$ if 
\begin{enumerate}
\item the map $X\to \Hom(\Gamma,\mathbf{C})$ defined by $x\to Z_x$ is continuous at $x=x_0$,
\item assume $Q_0$ is a quadratic form compatible with the stability data $(Z_{x_0}, a_{x_0})$, then there is an open neighborhood $\mathcal{U}$ of $x_0$ such that $(Z_x,a_x)$ are all compatible with $Q_0$ for all $x\in \mathcal{U}$. 
\item for any closed strict cone sector $V\in \mathcal{S}$ such that $Z(\Supp a_{x_0})\cap \partial V=\emptyset$, the map
\[
x\to \log A_{V,Z_x,Q_x}\in \mathfrak{g}_{V,Z_x, Q_x}\subset \prod_{\gamma\in \Gamma}\mathfrak{g}_\gamma,
\]

is continuous. Here $\prod_{\gamma\in\Gamma}\mathfrak{g}_\gamma$ has the product topology of the discrete topology. \label{continuous}
\end{enumerate}
\end{definition}

It is proved as a proposition in \cite{KS08} that there is a Hausdorff topology on $\Stab(\mathfrak{g})$ such that a continuous family as above is equivalent to a continuous map from $X$ to $\Stab(\mathfrak{g})$. The proposition also implies that the map $\Stab(\mathfrak{g})\to \Hom(\Gamma, \mathbf{R}^2)$ by $\sigma\mapsto Z$ is a local homeomorphism onto the image. \\

The most important property of the definition for our purpose is (\ref{continuous}), which implies that for all $\gamma\in \Gamma\backslash\{0\}$, the $\gamma$-component of $\log A_{V,Z_x,Q_x}$ stays constant, as long as no $Z(\gamma)$ with $a(\gamma)\neq 0$ enters into $V$. Denote the ray $\mathbf{R}_{\geqslant 0}Z_x(\gamma)$ by $l_{\gamma,x}$. Fix $\gamma\in \Gamma$, and focus on $V=l_{\gamma,x}$. Recall that the image $Z(\Supp a)$ is discrete. Then we arrive at the conclusion: if $l_{\gamma',x}=l_{\gamma,x}$ implies that $\gamma'$ is a multiple of $\gamma$, then $a_x(\gamma)$ is a constant in a neighborhood of $x$. \\

Once this property is understood, the WCF is obtained almost immediately. As in $\Stab(\mathcal{C})$, define the wall $\mathcal{W}$ in $\Stab(\mathfrak{g})$ to be the subset where two $\mathbf{Q}$-linearly independent vectors $\gamma_1,\gamma_2\in\Gamma$ are mapped to the same ray by $Z$. \\

Consider a continuous path $Z_t$ crossing the wall $\mathcal{W}$ at $t_0$. Assume the generic case: at $t_0\in \mathcal{W}$, two $\mathbf{Q}$-linearly independent primitive vectors $\gamma_1$ and $\gamma_2$ are mapped to the same ray $l$, and any other vector mapped to $l$ is generated by $\gamma_1$ and $\gamma_2$. Fix the rank $2$ lattice generated by $\gamma_1,\gamma_2$

\[
\Gamma_0:=\mathbf{Z}\gamma_1\oplus\mathbf{Z}\gamma_2. 
\]

The subset of primitive vectors is denoted by $\Gamma_0^{\prim}$.\\

 By Property (2) in the definition of the topology, we can fix a constant quadratic form $Q$ in a neighborhood of $t_0$ and $Q(\gamma_1)>0, Q(\gamma_2)>0$. Now assume $V_\epsilon$ be a small strict cone sector containing $l$, and no $Z(\gamma)$ with $a(\gamma)\neq 0$ crosses the boundary of $V_\epsilon$ in a neighborhood of $t_0$. Again such $V_\epsilon$ exists by the support property. By the Property (\ref{continuous}), $A_{V_\epsilon}$ stays constant in a neighborhood of $t_0$. However, $A_{V_\epsilon}$ has two different factorizations as $t\to t_0^-$ and $t\to t_0^+$. 
 \[
 \prod^{\longrightarrow}_{l\in V_\epsilon}A^-_{V_\epsilon}=\prod^{\longrightarrow}_{l\in V_\epsilon} A^+_{V_\epsilon}. 
 \]
 
 According to the analysis of Property (\ref{continuous}), if $\gamma\notin\Gamma_0$, $a(\gamma)$ is a constant near $t_0$. Furthermore, as these $\gamma$ are not mapped to $l$, they can be cancelled  in the above equality as $\epsilon\to 0$. Therefore, taking the limit, we have an equality involving only $\gamma\in \Gamma_0$. Assume $Z(\gamma_1)\wedge Z(\gamma_2)$ gives the normal orientation of $\mathbf{R}^2$ when $t<t_0$, and the orientation is changed when $t>t_0$. We get

\begin{proposition}[the wall--crossing formula]
 \[
 \prod^{\longrightarrow}_{\gamma\in \Gamma_0^{\prim}, Q(\gamma)>0} \exp{\Bigg(\sum_{n>0}a^-(n\gamma)\Bigg)}=\prod^{\longleftarrow}_{\gamma\in\Gamma_0^{\prim}, Q(\gamma)>0}\exp{\Bigg(\sum_{n>0}a^+(n\gamma)\Bigg)}. 
 \]
 
Here $a^-(\gamma)$ (resp., $a^+(\gamma)$) means $\lim_{t\to t^-}a_t(\gamma)$ (resp., $\lim_{t\to t^+}a_t(\gamma)$). 
\end{proposition}

A consequence of the WCF is that, we can lift a generic path $\{Z_t\}_{0\leqslant t \leqslant 1}$ in $\Hom(\Gamma,\mathbf{R}^2)$ to a unique continuous path $\{(Z_t,a_t)\}_{0\leqslant t\leqslant 1}$ in the space $\Stab(\mathfrak{g})$, starting at a given point $(Z_0,a_0)$. This is Theorem 3 in \cite{KS08} Sect. 2.3. \\

Here is an important example. Suppose $\Gamma$ is a free abelian group of finite rank, equipped with a skew-symmetric integer valued bilinear form $\langle\cdot,\cdot\rangle$. Define the graded $k$-vector space 
\[
\mathfrak{g}_\Gamma:=\bigoplus_{\gamma\in\Gamma} ke_\gamma, 
\]

and the bracket
\[
[e_{\gamma_1},e_{\gamma_2}]:=(-1)^{\langle \gamma_1,\gamma_2\rangle}\langle \gamma_1,\gamma_2\rangle e_{\gamma_1+\gamma_2}. 
\]

It is easy to check that $[\cdot,\cdot]$ defines a $\Gamma$-graded Lie algebra structure on $\mathfrak{g}_\Gamma$. This example is closely related to DT invariants for $3$-Calabi--Yau categories.\\
 
Suppose that $\mathcal{C}$ is a $3$-Calabi--Yau category, with a cyclic structure $(\cdot,\cdot)$. Let $\Gamma=N(\mathcal{C})$. Define $\langle \cdot, \cdot\rangle$ to be the Euler form
\[
\langle E,F\rangle:=\chi(E, F)
\]

Since $\mathcal{C}$ has the cyclic structure $(\cdot,\cdot)$ of degree $3$, $\langle\cdot,\cdot\rangle$ is skew-symmetric, and is bilinear. By the definition of $N(\mathcal{C})$, $\langle\cdot,\cdot\rangle$ is non-degenerate. Construct the $\Gamma$-graded Lie algebra $\mathfrak{g}_\Gamma$ as above. For any stability condition $\sigma=(Z,\mathcal{P})$ for $\mathcal{C}$, we want to associate a stability data for the Lie algebra $\mathfrak{g}_\Gamma$. Naturally, the central charge $Z$ should be the same. It follows that the walls for $\Stab(\mathcal{C})$ are the same as the walls for $\Stab(\mathfrak{g}_\Gamma)$. Let $\Omega(\gamma)$ be the DT invariants for $\sigma$. Define a map $f:\Stab(\mathcal{C})\to \Stab(\mathfrak{g}_\Gamma)$ by
\begin{equation}\label{f}
a(\gamma):= \sum_{n\geqslant 1, \gamma/n\in \Gamma\backslash\{0\}}-\frac{\Omega(\gamma/n)}{n^2}e_\gamma. 
\end{equation}

For the usual $3$-Calabi--Yau categories $\mathcal{C}$ and stability conditions, one is often offered a natural quadratic form $Q$ to show that $a$ has the support property. For example, if $\mathcal{C}=\mathcal{D}^b(X)$, we have the Hodge--Riemann bilinear form on the cohomology group. We assume that $a$ thus defined has the support property, and $f$ is then well defined. Now the statement that the DT invariants $\Omega(\gamma)$ satisfy WCF is equivalent to the statement that $f$ is continuous. Define the dilogarithm function
\[
\Li(t):=\sum_{m\geqslant 1}\frac{t^m}{m^2}. 
\] 

Assume $\gamma$ is primitive, and $l$ is the ray containing $Z(\gamma)$. 
\[
A_l=\exp{\Bigg(\sum_{n\geqslant 1} a(n\gamma)\Bigg)}=\exp{\Bigg(-\sum_{n\geqslant 1}\Omega(n\gamma)\sum_{m\geqslant 1} \frac{e_{mn\gamma}}{m^2}\Bigg)}=\exp{\Bigg(-\sum_{n\geqslant 1}\Omega(n\gamma)\Li(e_{n\gamma})\Bigg)}. 
\]

Therefore, if $f$ is continuous over the small interval $(t_0^-,t_0^+)$, we get the WCF for DT invariants
\[
 \prod^{\longrightarrow}_{\gamma\in \Gamma_0^{\prim}, Q(\gamma)>0} \exp{\Bigg(\sum_{n\geqslant 1}\Omega^-(n\gamma)\Li(e_{n\gamma})\Bigg)}=\prod^{\longleftarrow}_{\gamma\in\Gamma_0^{\prim}, Q(\gamma)>0}\exp{\Bigg(\sum_{n\geqslant 1}\Omega^+(n\gamma)\Li(e_{n\gamma})\Bigg)}. 
\]

We can ask the same question in another way. Consider the unique continuous lift $\{(Z_t,a_t)\}_{0\leqslant t\leqslant 1}$ in the space $\Stab(\mathfrak{g})$, starting at a given point $(Z_0,a_0)$. If we define $\Omega_\sigma(\gamma)$ by Equation~\eqref{f}, are $\Omega_\sigma(\gamma)$ DT invariants for the category $\mathcal{C}$? WCF almost forces us to define DT invariants this way. \\

The definition of motive DT invariants and the motivic WCF are jointly expressed in terms of a continuous map (local homeomorphism) from $\Stab(\mathcal{C}_V)$ to some motivic quantum tori. This is the main theorem, Theorem 7 of \cite{KS08} Sect. 6.2. The proof is highly nontrivial, and we are not able to explain the ideas of the proof. 

\section{Interpretation as identities in the automorphism group of a torus}\label{3}
The Lie algebra $\mathfrak{g}_\Gamma$ is very special. Let us introduce a commutative associative product on $\mathfrak{g}_\Gamma$ by
\[
e_{\gamma_1}\cdot e_{\gamma_2}:=(-1)^{\langle\gamma_1,\gamma_2\rangle}e_{\gamma_1+\gamma_2}. 
\]

The result $k$-algebra is a twist of the usual group algebra for an algebraic torus. Define $\mathbb{T}_\Gamma$ to be the spectrum of this commutative associative algebra. It is the quasi-classic limit of the quantum torus we mentioned in Sect.~\ref{1}. It is a torsor over the algebraic torus $\Hom(\Gamma, \mathbb{G}_m)$, and is also called a torus. $\mathbb{T}_\Gamma$ is an algebraic Poisson manifold with the Poisson bracket
\[
\{f,g\}:=[f,g]. 
\]
  
The Lie algebra $\mathfrak{g}_\Gamma$ now acts on $\mathbb{T}_\Gamma$ by Hamiltonian vector fields. Denote by $\theta_\gamma$ the formal Poisson automorphism on $\mathbb{T}_\Gamma$
\[
\theta_\gamma:=\exp{\Big(\big\{-\Li(e_\gamma),\cdot\big\}\Big)}. 
\]

Compute it on the basis
\begin{equation}\label{intrinsic formula}
\theta_\gamma(e_\mu)=(1-e_\gamma)^{\langle \gamma,\mu\rangle}e_\mu. 
\end{equation}

Consider the formal automorphisms, we can write
\[
A_l:=\prod^{\longrightarrow}_{Z(\gamma)\in l} \theta_\gamma^{\Omega(\gamma)}. 
\]

Therefore WCF are identities in the formal automorphism group of $\mathbb{T}_\Gamma$. \\

Consider the wall crossing at a generic point $t_0$ of the wall $\mathcal{W}$. We can restrict everything to the rank $2$ sublattice $\Gamma_0$. Assume $\langle \gamma_1,\gamma_2\rangle=k>0$. Since $\langle\cdot,\cdot\rangle$ on $\Gamma_0$ is non-degenerate, the torus $\mathbb{T}_{\Gamma_0}$ is a symplectic manifold with the symplectic form $-k^{-1}(xy)^{-1} \ud x\wedge\ud y$. We can write the formula \eqref{intrinsic formula} in terms of the basis $\gamma_1,\gamma_2$ and define $\theta_{(a,b)}^{(k)}$, a formal automorphism of $\mathbb{T}_{\Gamma_0}^{(k)}$. However, by a choice of a quadratic refinement, which we will explain later, we can identify the twisted torus $\mathbb{T}_{\Gamma_0}$ with an ordinary algebraic torus $\mathbb{T}^{(k)}$ with the symplectic form $-k^{-1}(xy)^{-1}\ud x\wedge\ud y$. Then the formula is defined by
\begin{equation}\label{untwist formula}
\theta_{(a,b)}^{(k)}(x):=x(1-(-1)^{kab}x^ay^b)^{-kb}, \quad \theta_{(a,b)}^{(k)}(y):=y(1-(-1)^{kab}x^ay^b)^{ka}.
\end{equation}

\begin{remark}
Notice the difference between Formula \eqref{intrinsic formula} and Formula \eqref{untwist formula} is the sign $(-1)^{kab}$. This is the cost we have to pay if we want to write WCF as identities of the automorphism group of an ordinary torus. Therefore, the use of the twist torus is really the way to make WCF simpler. However, the formula in terms of ordinary torus automorphisms is usually what people use. And it is also the form used in the tropical vertex group. 
\end{remark}

Consider the formal automorphism group generated by these elements. Each element $\theta$ in this group can be factorized either in the clockwise order or in the counter clockwise order. Clockwise means that the slope $b/a$ is decreasing
\[
\theta=\prod^{\longrightarrow}_{(a,b)\in\Gamma^{\prim}_0}\Big(\theta_{(a,b)}^{(k)}\Big)^{c^-_{a,b}}=\prod^{\longleftarrow}_{(a,b)\in\Gamma^{\prim}_0}\Big(\theta_{(a,b)}^{(k)}\Big)^{c^+_{a,b}}.
\]

These identities are also called WCF. A priori , the numbers $c_{a,b}$ are in $\mathbf{Q}$, and do not necessarily come from any stability conditions. However, if we choose $\theta$ to be some special commutator, the numbers $c_{a,b}$ are integers and are DT invariants.\\

As an example, assume $k=1$, define $S=\theta_{(1,0)}^{(1)}$ and $T=\theta_{(0,1)}^{(1)}$. Let $\theta$ to be the commutator 
\[
T^{-1}\circ S\circ T\circ S^{-1}=\prod^{\longrightarrow}_{(a,b)\in\Gamma^{\prim}_0}\Big(\theta_{(a,b)}^{(1)}\Big)^{c_{a,b}}.
\]

It is proved in \cite{Re} that $c_{a,b}$ are integers determined by the Euler characteristic of framed moduli spaces of semi-stable representations of quivers. Therefore this is a WCF for DT invariants.\\

These factorization formulas for commutators also appear in the tropical vertex group. In order to make the relation more explicitly, we do the following embedding. Recall if $(\mathfrak{g}_\Gamma, \langle\cdot,\cdot\rangle)$ is from a $3$-Calabi-Yau category $\mathcal{C}$, the skew-symmetric bilinear form $\langle\cdot,\cdot\rangle$ is non-degenerate. Therefore it induces an isomorphism from $\Gamma$ to its dual $\Gamma^\vee$. Since $\Gamma_0$ is of rank $2$, let's denote the image of $\gamma\in \Gamma_0$ by $\gamma^\perp$.
\[
\gamma^\perp:=\langle \gamma,\cdot\rangle\in \Gamma_0^\vee.
\]

Consider the lattice $\Lambda:=\Gamma_0\oplus\Gamma_0^\vee$, with a skew-symmetric bilinear pairing $\langle\cdot,\cdot\rangle$:
\[
\langle (\gamma_1,\nu_1),(\gamma_2,\nu_2)\rangle:=-\langle\gamma_1,\gamma_2\rangle +\nu_1(\gamma_2)-\nu_2(\gamma_1).
\]

Then we embed the $(\Gamma_0,\langle\cdot,\cdot\rangle)$ into $(\Lambda, \langle\cdot,\cdot\rangle)$ by 
\[
\gamma\mapsto (\gamma,\gamma^\perp). 
\]

The lattice $(\Lambda, \langle\cdot,\cdot\rangle)$ defines a symplectic manifold called the symplectic double torus. It has the Poisson structure such that the Lie algebra $\mathfrak{g}_{\Gamma_0}$ is contained as a Lie subalgebra, and the Lie subalgebra corresponding to $\Gamma_0^\vee$ is abelian. \\

Similarly, in order to construct the tropical vertex group, we begin with a lattice $M$ of rank $2$ and its dual $N$. Construct a larger Lie algebra $k[M]\otimes_\mathbf{Z} N$ from the lattice $M\oplus N$. In terms of the group ring, one writes
\[
e_{(m,n)}=z^m\partial_n. 
\]

The Lie algebra is defined by
\[
[z^{m_1}\partial_{n_1},z^{m_2}\partial_{n_2}]:=z^{m_1+m_2}\partial_{n_1(m_2)\cdot n_2-n_2(m_1)\cdot n_1}. 
\]

Instead of making a pronilpotent group for each strict cone sector, in \cite{GPS}, the Lie algebra over $k$ is tensored with an Artin local ring or a complete local ring $R$ with the maximal ideal $\mathfrak{m}_R$. Define the Lie algebra
\[
\mathfrak{g}_R:=\mathfrak{m}_R\hat{\otimes}_kk[M]\otimes_\mathbf{Z}N.
\]

Since $\mathfrak{g}_R$ is complete with respect to $\mathfrak{m}_R$-adic topology, there is a pronilpotent Lie group with the Lie algebra $\mathfrak{g}_R$. Define the Lie subalgebra $\mathfrak{h}_R\subset\mathfrak{g}_R$ to be
\[
\mathfrak{h}_R:=\bigoplus_{m\in M\backslash\{0\}}z^m(\mathfrak{m}_R\otimes m^\perp). 
\]  

The tropical vertex group $\mathbb{V}_R$ is defined to be the Lie subgroup corresponding to $\mathfrak{h}_R$. Write out the Lie bracket for $\mathfrak{h}_R$,
\begin{align*}
[z^{\gamma_1}\partial_{\gamma_1^\perp},z^{\gamma_2}\partial_{\gamma_2^\perp}]&=z^{\gamma_1+\gamma_2}\partial_{\gamma_1^\perp(\gamma_2)\cdot \gamma_2^\perp-\gamma_2^\perp(\gamma_1)\cdot \gamma_1^\perp}\\
&=z^{\gamma_1+\gamma_2}\partial_{\langle \gamma_1,\gamma_2\rangle(\gamma_2^\perp+\gamma_1^\perp)}\\
&=\langle \gamma_1,\gamma_2\rangle z^{\gamma_1+\gamma_2}\partial_{(\gamma_1+\gamma_2)^\perp}. 
\end{align*}

Here we also define $\langle\gamma_1,\gamma_2\rangle:=\gamma_1^\perp(\gamma_2)$. This is a non-degenerate skew-symmetric pairing. \\

Recall the Lie bracket for  $\mathfrak{g}_{\Gamma_0}$,
\[
[e_{(\gamma_1,\gamma_1^\perp)},e_{(\gamma_2,\gamma_2^\perp)}]=(-1)^{\langle \gamma_1,\gamma_2\rangle}\langle \gamma_1,\gamma_2\rangle e_{(\gamma_1+\gamma_2,(\gamma_1+\gamma_2)^\perp)}. 
\]

We can compare this Lie algebra $\mathfrak{h}_R$ with $\mathfrak{g}_{\Gamma_0}$ we had before. Choose $(R,\mathfrak{m}_R)$ to be the toric algebra from a strict cone sector, and identify the two skew-symmetric pairing $\langle \cdot,\cdot\rangle$, we find $\mathfrak{g}_{\Gamma_0}$ and $\mathfrak{h}_R$ are almost the same except for a sign $(-1)^{\langle\gamma_1,\gamma_2\rangle}$. Define the map $f_\sigma: \mathbb{T}_{\Gamma_0}\to \mathbb{T}$ by 
\[
z^\gamma\mapsto \sigma(\gamma)e_\gamma,
\]

where $\mathbb{T}$ is an ordinary algebraic torus defined by $k[M]$. $f_\sigma$ is a homomorphism of algebras if $\sigma: \Gamma_0\to \pm 1$ satisfies
\[
\sigma(\gamma_1)\sigma(\gamma_2)=(-1)^{\langle\gamma_1,\gamma_2\rangle}\sigma(\gamma_1+\gamma_2).
\]

The function $\sigma$ is the quadric form associated to the symmetric bimultiplical form $(-1)^{\langle\gamma_1,\gamma_2\rangle}$. Such a function $\sigma$ is determined by the values on a basis of $\Gamma_0$. Therefore, it aways exists, but not unique. A choice of such a $\sigma$ is called the quadratic refinement. It identifies the twisted torus $\mathbb{T}_{\Gamma_0}$ with an ordinary torus.\\

Extend $f_\sigma$ to 
\[
z^\gamma\partial_{\gamma^\perp}\mapsto \sigma(\gamma)e_{(\gamma,\gamma^\perp)}.
\]

The Lie algebras $\mathfrak{h}_R$ and $\mathfrak{g}_{\Gamma_0}$ are thus identified, and the identities in \cite{GPS} are the same with the WCF introduced here. \\

The invariants involved in the tropical vertex groups are relative Gromov--Witten invariants. See \cite{GPS} and the lecture notes of Sara Filippini's in this volume. The close relationship between these two sides are part of the big picture we emphasized at the beginning. Again interested readers should turn to Kontsevich and Soibelman's work \cite{KS13}. 
\bibliographystyle{amsalpha}
\bibliography{WCF}

\providecommand{\bysame}{\leavevmode\hbox to3em{\hrulefill}\thinspace}
\providecommand{\MR}{\relax\ifhmode\unskip\space\fi MR }
\providecommand{\MRhref}[2]{%
  \href{http://www.ams.org/mathscinet-getitem?mr=#1}{#2}
}
\providecommand{\href}[2]{#2}
\begin{thebibliography}{GMN11}

\bibitem[Bay09]{Ba09}
A.~Bayer, \emph{Polynomial {B}ridgeland {S}tability {C}onditions and the
  {L}arge {V}olume {L}imit}, Geometry and Topology \textbf{13} (2009),
  2389--2425.

\bibitem[Beh09]{Br}
K.~Behrend, \emph{{D}onaldson--{T}homas type invariants via microlocal
  geometry}, Annals of Mathematics \textbf{170} (2009), no.~3, 1307--1338.

\bibitem[Bri07]{Br07}
T.~Bridgeland, \emph{Stability {C}onditions on {T}riangulated {C}ategories},
  Annals of Mathematics \textbf{166} (2007), 317--345.

\bibitem[CV93]{CV}
S.~Cecotti and C.~Vafa, \emph{On {C}lassification of {$N=2$} {S}upersymmetric
  {T}heories}, Commun. Math. Phys. \textbf{158} (1993), 569--644.

\bibitem[DT98]{DT98}
S.~Donaldson and R.~Thomas, \emph{{G}auge {T}heory in {H}igher {D}imensions},
  The Geometric universe: science, geometry, and the work of Roger Penrose
  (S.~Huggett, ed.), Oxford University Press, 1998.

\bibitem[GMN11]{GMN}
D.~Gaiotto, G.~Moore, and A.~Neitzke, \emph{{W}all--crossing in {C}oupled
  $2d-4d$ {S}ystems}, arXiv 1103.2598, 2011.

\bibitem[GPS10]{GPS}
M.~Gross, R.~Pandharipande, and B.~Siebert, \emph{The {T}ropical {V}ertex},
  Duke Math. J. \textbf{153} (2010), no.~2, 297--362.

\bibitem[JS12]{JS}
D.~Joyce and Y.~Song, \emph{A {T}heory of {G}eneralized {D}onaldson--{T}homas
  {I}nvariants}, vol. 217, Memoirs of the American Mathematical Society, no.
  1020, American Mathematical Society, 2012.

\bibitem[Kel08]{Ke}
B.~Keller, \emph{{C}alabi--{Y}au {T}riangulated {C}ategories}, Trends in
  Representation Theory of Algebras (A.~Skowro\'{n}ski, ed.), European
  Mathematical Society, Zurich, 2008.

\bibitem[KS06]{KS04}
M.~Kontsevich and Y.~Soibelman, \emph{Affine {S}tructures and
  {N}on-{A}rchimediean {A}nalytic {S}paces}, The Unity of Mathematics: in honor
  of the Ninetieth birthday of I.M. Gelfand (P.~Etingof, V.~Retakh, and I.M.
  Singer, eds.), Progress in Mathematics, vol. 244, Birkhauser, 2006,
  pp.~321--385.

\bibitem[KS08]{KS08}
\bysame, \emph{Stability {S}tructures, {M}otivic {D}onaldson--{T}homas
  {I}nvariants and {C}luster {T}ransformations}, arXiv:0811.2435 [math.AG],
  2008.

\bibitem[KS10]{KS10}
M.~Kontsevich and Y.~Soibelman, \emph{Motivic {D}onaldson--{T}homas invariants:
  summary of results}, arXiv: 0910.4315, 2010.

\bibitem[KS11]{KS11}
M.~Kontsevich and Y.~Soibelman, \emph{Lectures on {M}otivic
  {D}onaldson--{T}homas {I}nvariants and {W}all--crossing {F}ormulas}, online
  notes, 2011.

\bibitem[KS13]{KS13}
\bysame, \emph{Wall--crossing {S}tructures in {D}onaldson--{T}homas
  {I}nvariants, {I}ntegrable {S}ystems and {M}irror {S}ymmetry},
  arXiv:1303.3253 [math.AG], 2013.

\bibitem[Nei]{Nslides1}
A.~Neitzke, \emph{A {W}all--crossing {F}ormula for $2d-4d$ {DT} invariants}.

\bibitem[Rei10]{Re}
M.~Reineke, \emph{Poisson {A}utomorphisms and {Q}uiver {M}oduli}, Journal of
  the Institute of Mathematics of Jussieu \textbf{9} (2010), 653--667.

\bibitem[Tho00]{T00}
R.~Thomas, \emph{A {H}olomorphic {C}asson {I}nvariant for {C}alabi--{Y}au
  {$3$}-folds and {B}undles on {$K3$} {F}ibrations}, J. Differential Geom.
  \textbf{54} (2000), 367--438.

\bibitem[Tod]{To09}
Y.~Toda, \emph{Limit {S}table {O}bjects on {C}alabi--{Y}au $3$-folds},
  arXiv:0803.2356 [math. AG].

\end{thebibliography}

\end{document}